\documentclass[12pt]{article}

\DeclareMathAlphabet{\mathpzc}{OT1}{pzc}{m}{it}

\usepackage{amsmath}
\usepackage{amssymb}
\usepackage{amsfonts}
\usepackage{amsthm}
\usepackage{mathrsfs}
\usepackage{fontenc}
\usepackage{textcomp}
\usepackage{ifsym}
\usepackage{enumerate}
\usepackage[arrow,curve]{xy}
\usepackage[OT2,T1]{fontenc}
\DeclareSymbolFont{cyrletters}{OT2}{wncyr}{m}{n}
\DeclareMathSymbol{\Sha}{\mathalpha}{cyrletters}{"58}

\DeclareMathOperator{\Col}{Col}
\DeclareMathOperator{\Nul}{Nul}

\newcommand{\F}{\mathbb{F}}

\newcommand{\HLINE}{\hbox to\hsize{\hrulefill}}

\allowdisplaybreaks

\newtheorem*{thm}{Theorem}

\newtheorem*{cor}{Corollary}

\begin{document}

\title{Symmetric Matrices over $\F_2$ and the Lights Out Problem}
\author{Igor Minevich}
\maketitle
\begin{abstract}
We prove that the range of a symmetric matrix over $\F_2$ contains the vector of its diagonal elements. We apply the theorem to a generalization of the ``Lights Out'' problem on graphs.
\end{abstract}

\begin{section}{Introduction}
The purpose of this note is to prove that the range of a symmetric matrix over $\F_2$ contains the vector of its diagonal elements; we do this in section \ref{sect:linalg}. We then apply this theorem to a generalization of the ``Lights Out'' problem on graphs: we have a simple graph and the vertices all have a state of either `on' (represented by $0 \in \F_2$) or `off' (represented by $1 \in \F_2$). When a vertex is `clicked,' the states of it and its adjacent vertices are toggled. Starting with arbitrary states for all vertices, the problem is to determine whether or not we can click a subset of vertices and get to the state where they are all off, hence the name ``Lights Out''. It is known that if we start with all vertices on, then this is possible (see ~\cite{MR1386937} or ~\cite{MR987520}). The theorem we prove here generalizes this result to the case where clicking some of the vertices does not actually affect the state of those vertices, but only the states of those adjacent to them.
\end{section}

\begin{section}{Linear Algebra Theorem}
\label{sect:linalg}
\begin{thm}
Any symmetric $N \times N$ matrix $A = (a_{ij})$ over $\F_2$ has its diagonal vector $\vec d = (a_{11}, a_{22}, \dots, a_{NN})^T$ in its range.
\end{thm}
\begin{proof}
We show that if $A\vec x = \vec 0$ then $\vec x \cdot \vec d = 0$. This implies $\vec d$ is perpendicular to every vector in $\Nul A$, i.e. $\vec d \in \Nul(A)^\perp = \Col(A^T) = \Col(A)$, as desired. If $\vec x = \vec 0$, then this is obvious. Without loss of generality, we may assume $\vec x = (1, \dots, 1, 0, \dots, 0)$, where the first $n+1$ entries are $1$ and $n \ge 0$ (otherwise, we can just permute the entries of $\vec x$ and also the corresponding columns and rows of $A$ so that this is true). The linear dependence relation on the columns of $A$ given by $A\vec x = \vec 0$ says that the $(n+1)$st column is the sum of the first $n$ columns, so we have
\[a_{j, n+1} = a_{jj} + \sum_{\substack{i=1\\i \ne j}}^{n} a_{ji}, \qquad j = 1, \dots, N\]
\[a_{n+1, n+1} = \sum_{j=1}^{n} a_{n+1, j} = \sum_{j=1}^{n} a_{j, n+1} =\]
\[= \sum_{j=1}^{n} a_{jj} + \sum_{j=1}^{n}\sum_{\substack{i=1\\i \ne j}}^{n} a_{ji} = \sum_{j=1}^n a_{jj} + 2\sum_{n\ge i>j\ge 1} a_{ij} = \sum_{j=1}^n a_{jj}\]
so we see that $\vec x \cdot \vec d = \sum_{j=1}^{n+1} a_{jj} = 0$.
\end{proof}
\end{section}

\begin{section}{``Lights Out'' Problem}
\label{sect:LightsOut}
The original motivation for this theorem came from the ``Lights Out'' problem as explained in the introduction.
\begin{cor}
Let $G$ be a finite undirected graph, where there can be at most one edge between any two vertices and a vertex can be connected to itself. Let $v_1, \dots, v_n$ be the vertices connected to themselves. Assign to each vertex a state of either $1$ or $0$, and assume that clicking a vertex makes those vertices adjacent to it (including itself if it is connected to itself) switch their state. Then, starting with all vertices at the state of 0, we can choose vertices to click so that, as a result, precisely $v_1, \dots, v_n$ are 1 and the rest are 0.
\label{cor}
\end{cor}
\begin{proof}
Define the $N \times N$ matrix $A$ over $\F_2$, where $N$ is the number of vertices of $G$, by $a_{ij} = 1$ if and only if the vertex $v_i$ is connected to the vertex $v_j$. Since the graph is undirected, $A$ is symmetric. Furthermore, the first $n$ diagonal elements of $A$ are $1$ and the rest are $0$, since the first $n$ vertices are the ones connected to themselves. It is easy to see that if $A\vec x = \vec y$, then $\vec y$ is precisely the vector of states obtained by starting with all vertices being 0 and clicking the vertices corresponding to 1's in $\vec x$, because multiplication by $\vec x$ simply adds up the columns corresponding to 1's in $\vec x$. Since the diagonal vector $\vec d$ is in the range of $A$, there is a vector $\vec x$ such that $A \vec x = \vec d$. Thus we can click the vertices corresponding to 1's in $\vec x$ and, as a result, the vertices that are connected to themselves will be $1$ and the rest will be $0$, as desired. 
\end{proof}

This is the generalization of the ``Lights Out'' problem on graphs to the case where not all vertices affect themselves. To make this more concrete, we first list a few examples where all vertices are connected to themselves and then an example where they are not. The original motivator for the problem was the case where the vertices of the graph are squares in a finite two-dimensional grid and clicking a square toggles its color (say between white and black) and also the colors of the squares that share an edge with the square clicked. The problem is to find out which squares to click so that, if we start with an all-white grid, we end up with all squares black. This is not easy to do by inspection for grids $5 \times 5$ and bigger, but the corollary shows it can always be done.\\
\\
We could generalize this scenario to be $n$-dimensional, where an $n$-dimensional cube affects itself and those with which it shares an $(n-1)$-dimensional hyperplane. Or we could start with an arbitrary subset of a 2-dimensional or 3-dimensional grid, perhaps in a pleasing shape such as a star or a flower, and have vertices affect themselves and the ones next to them as before. We could also change the rules and have squares affect each other diagonally, or in a torus-like fashion, where the top squares affect the bottom ones and/or the leftmost ones affect the rightmost ones (this particular case has been studied in more detail in \cite{MR2590299}). Or we could change `squares' to triangles in a triangular lattice or hexagons in a hexagonal lattice, etc.\\
\\
Now for an example of the generalization, where some vertices do affect themselves and some don't, we can have squares with green and red lamps; the green lamps affect themselves and also those around them, but the red lamps only affect those around them and not themselves. Starting with all lamps off, the goal is to turn on all the green lamps and turn off all the red ones, and by the corollary this can always be achieved. Of course, the amount of such scenarios is limited only by the imagination.
\end{section}

\bibliography{mybib}

\begin{thebibliography}{Car96}

\bibitem[Car96]{MR1386937}
Yair Caro.
\newblock Simple proofs to three parity theorems.
\newblock {\em Ars Combin.}, 42:175--180, 1996.

\bibitem[GY09]{MR2590299}
Masato Goshima and Masakazu Yamagishi.
\newblock Two remarks on torus lights out puzzle.
\newblock {\em Adv. Appl. Discrete Math.}, 4(2):115--126, 2009.

\bibitem[Sut88]{MR987520}
Klaus Sutner.
\newblock Additive automata on graphs.
\newblock {\em Complex Systems}, 2(6):649--661, 1988.

\end{thebibliography}
\bibliographystyle{alpha}

\end{document}